\documentclass[amstex,12pt, amssymb]{article}

\usepackage{mathtext}
\usepackage[cp1251]{inputenc}
\usepackage[T2A]{fontenc}
\usepackage[dvips]{graphicx}
\usepackage{amsmath}
\usepackage{amssymb}
\usepackage{amsxtra}
\usepackage{latexsym}
\usepackage{ifthen}

\newcounter{lemma}[section]

\newcounter{corollary}[section]

\newcounter{remark}[section]

\newcounter{theorem}[section]

\newcounter{proposition}[section]

\newcounter{example}

\numberwithin{equation}{section}

\begin{document}

\markboth{E.\ A.\ Sevost'yanov}{\centerline{ON MAPPINGS WITH THE
INVERSE POLETSKY INEQUALITY ...}}

\def\cc{\setcounter{equation}{0}
\setcounter{figure}{0}\setcounter{table}{0}}

\overfullrule=0pt

%\normalsize\large

\author{E.\ A.\ Sevost'yanov}

\title{
{\bf ON MAPPINGS WITH THE INVERSE POLETSKY INEQUALITY ON RIEMANNIAN
SURFACES}}

\date{\today}
\maketitle

%\large
\begin{abstract}
We study some problems related to the boundary behavior of maps of
domains of Riemannian surfaces. In particular, for mappings
satisfying the inverse Poletsky type modulus inequality, we
establish the possibility of their continuous extension to the
boundary in terms of prime ends. We also study the local behavior of
such mappings at boundary points.
\end{abstract}


\section{Introduction}

This article is devoted to the study of mappings with bounded and
finite distortion, which have been actively studied recently, see,
for example, \cite{IM} and \cite{MRSY}. In a relatively recent
series of our papers, we carried out a detailed study of mappings
satisfying the so-called inverse Poletsky inequality. Note that we
have given a more or less complete description of this class both in
Euclidean and in metric spaces, see, for example, \cite{SevSkv$_2$},
\cite{SF} and \cite{Sev$_3$}. This article is devoted to another
important aspect of the study of these mappings, namely, we study
mappings which are defined in domains on Riemannian surfaces with
bad boundaries. We are interested here in mappings with branch
points, since the corresponding case of homeomorphisms was studied
earlier (see~\cite{RV$_2$}, \cite{SDIK}). It should be noted that
the approach to prime ends on Riemannian surfaces is somewhat
different from the more conventional approach on Riemannian
manifolds and the Euclidean case, see e.g.~\cite{KR$_2$},
\cite{IS$_2$}. We also study the issues of the local behavior of
these mappings at points on the boundary of the domain.

\medskip
Most of the definitions found in the text can be found in~\cite{RV}.
Everywhere below, unless otherwise stated, the Riemannian surfaces
${\Bbb S}$ and ${\Bbb S}_*$ are of the hyperbolic type. Further
$ds_{\widetilde{h}}$ and $d\widetilde{v},$ $ds_{\widetilde{h_*}}$
and $d\widetilde{v_*}$ denote the length and area elements on
Riemannian surfaces $ {\Bbb S}$ and ${\Bbb S}_*,$ respectively. We
also use the notation $\widetilde{h}$ for the metric on the surface
${\Bbb S},$ in particular,
\begin{equation}\label{eq4}
\widetilde{B}(p_0, r):=\{p\in {\Bbb S}: \widetilde{h}(p, p_0)<r\},
\quad\widetilde{S}(p_0, r):=\{p\in {\Bbb S}: \widetilde{h}(p,
p_0)=r\}
\end{equation}
are the disk and the circle on ${\Bbb S}$ centered at $p_0$ and of
the radius $r>0,$ respectively. Set
\begin{equation}\label{eq1**}
\widetilde{A}=\widetilde{A}(p_0, r_1, r_2)=\{p\in {\Bbb S}:
r_1<\widetilde{h}(p, p_0)<r_2\}\,.
\end{equation}
The following definitions refer to Carath\'{e}odory~\cite{Car},
cf.~\cite{KR$_2$}. Recall that a continuous mapping $\sigma:{\Bbb
I}\rightarrow {\Bbb S},$ ${\Bbb I}=(0, 1),$ is called a {\it Jordan
arc} in ${\Bbb S},$ if $\sigma(t_1)\ne\sigma(t_2)$ for $t_1\ne t_2$.
Further we will sometimes use $\sigma$ for $\sigma({\Bbb I}),$
$\overline{\sigma}$ for $\overline{\sigma({\Bbb I})}$ and
$\partial\sigma$ for $\overline{\sigma({\Bbb
I})}\setminus\sigma({\Bbb I})$. A {\it cut} in a domain $D$ is
either a Jordan arc $\sigma:{\Bbb I}\rightarrow D,$ ends which lie
on $\partial D,$ or a closed Jordan curve in $D.$ The sequence
$\sigma_1,\sigma_2,\ldots,\sigma_m,\ldots$ of cuts in $D$ is called
a {\it chain} if:

\medskip
(i) $\overline{\sigma_i}\cap\overline{\sigma_j}=\varnothing$ for any
$i\ne j$, $i,j= 1,2,\ldots$;

\medskip
(ii)  $\sigma_m$ separates $D$, i.e., $D\setminus \sigma_m$ consists
precisely from two components one of which contains $\sigma_{m-1},$
and another contains $\sigma_{m+1},$

\medskip
(iii) $\widetilde{h}(\sigma_m)\rightarrow\infty$ as
$m\rightarrow\infty,$
$\widetilde{h}(\sigma_m)=\sup\limits_{p_1,p_2\in\sigma_m
}\widetilde{h}(p_1, p_2).$

\medskip
By the definition, a chain of cuts $\{\sigma_m\}$ defines a chain of
domains $d_m\subset D$ such that $\partial\,d_m\cap
D\subset\sigma_m$ and $d_1\supset d_2\supset\ldots\supset
d_m\supset\ldots$. Two chains of cuts $\{\sigma_m\}$ and
$\{\sigma_k^{\,\prime}\}$ are called {\it equivalent,} if for each
$m=1,2,\ldots$ the domain $d_m$ contains all the domains
$d_k^{\,\prime }$ except for a finite number, and for each
$k=1,2,\ldots$ the domain $d_k^{\,\prime}$ also contains all the
domains $d_m$ except for a finite number. {\it End} of $D$ is the
class of equivalent chains of cuts in $D$.

Let $K$ be a prime end in $D\subset {\Bbb S},$ and $\{\sigma_m\}$
and $\{\sigma_m^{\,\prime}\}$ are two chains in $K$, $d_m$ and
$d_m^{\,\prime}$ are domains corresponding to $\sigma_m$ and
$\sigma_m^{\,\prime}$. Then
$$\bigcap\limits_{m=1}\limits^{\infty}\overline{d_m}\subset
\bigcap\limits_{m=1}\limits^{\infty}\overline{d_m^{\,\prime}}\subset
\bigcap\limits_{m=1}\limits^{\infty}\overline{d_m}\ ,$$ and thus
$$\bigcap\limits_{m=1}\limits^{\infty}\overline{d_m}=
\bigcap\limits_{m=1}\limits^{\infty}\overline{d_m^{\,\prime}}\ ,$$
in other words, the set
$$I(K)=\bigcap\limits_{m=1}\limits^{\infty}\overline{d_m}$$
depends only on $K$ and does not depend on the choice of the chain
of cuts $\{\sigma_m\}$. The set $I(K)$ is called the {\it impression
of a prime end} $K$. Further $E_D$ denotes the set of all prime ends
in $D,$ and $\overline{D}_P:=D\cup E_D$ denotes the completion of
$D$ by its prime ends. Let us turn $\overline{D}_P$ into the
topological space as follows. Firstly, open sets from $D$ are
considered open in $\overline{D}_P,$ as well. Secondly, the base
neighborhood of a prime end $P\subset E_D$ is defined as the union
of an arbitrary domain $d,$ included in some chain of cuts of $P,$
with all other prime ends in $d.$ In particular, in the topology
mentioned above, a sequence of points $x_n\in D$ converges to an
element $P\in E_D$ if and only if, for any domain $d_m,$ belonging
to the chain of domains $d_1, d_2, d_3, \ldots,$ in $P$ there exists
$n_0=n_0(m) $ such that $x_n\in d_m$ for $n\geqslant n_0.$

\medskip
As usual, a path $\gamma$ on the Riemannian surface ${\Bbb S}$ is a
continuous mapping $\gamma:I\rightarrow {\Bbb S},$ where $I$ is a
finite segment, an interval or a half-interval of a real axis. Let
$\Gamma$ be a family of paths in ${\Bbb S}.$ A Borel function
$\rho:{\Bbb S}\rightarrow [0, \infty]$ is called {\it admissible}
for the family $\Gamma$ of paths $\gamma,$ if
$\int\limits_{\gamma}\rho(p)\,ds_{\widetilde{v}}(p)\geqslant 1$ for
any path $\gamma \in \Gamma.$ The latter is briefly written in the
form: $\rho\in {\rm adm}\,\Gamma.$  A {\it modulus} of the family
$\Gamma$ is a real-valued function
\begin{equation}\label{eq1A}
M(\Gamma):=\inf\limits_{\rho\in {\rm adm}\,\Gamma}\int\limits_{\Bbb
S}\rho^2(p)\,d\widetilde{v}(p)\,.
\end{equation}
Let $E,$ $F\subset {\Bbb S}$ be arbitrary sets. In the future,
everywhere by $\Gamma(E, F, D)$ we denote the family of all paths
$\gamma:[a,b]\rightarrow D,$ which join $E$ and $F$ in $D,$ that is,
$\gamma(a)\in E,$ $\gamma(b) \in F$ and $\gamma(t)\in D$ for $t\in
(a, \, b).$ Let $f:D\rightarrow {\Bbb S}_*,$ let $p_0\in f(D)$ and
let $0<r_1<r_2<d_0=\sup\limits_{p\in f(D)}\widetilde{h_*}(p, p_0).$
Now, we denote by $\Gamma_f(p_0, r_1, r_2)$ the family of all paths
$\gamma$ in $D$ such that $f(\gamma)\in \Gamma(S(p_0, r_1), S(p_0,
r_2), \widetilde{A}(p_0,r_1,r_2)).$
We say that {\it $f:D\rightarrow {\Bbb S}_*$ satisfies the inverse
Poletsky inequality} at a point $p_0\in \overline{f(D)},$ if the
relation
\begin{equation}\label{eq2*A}
M(\Gamma_f(p_0, r_1, r_2))\leqslant \int\limits_{f(D)\cap
\widetilde{A}(p_0, r_1, r_2)} Q(p)\cdot \eta^{2}(\widetilde{h_*}(p,
p_0))\, d\widetilde{v_*}(p)
\end{equation}
holds for any Lebesgue measurable function $\eta:
(r_1,r_2)\rightarrow [0,\infty ]$ such that
\begin{equation}\label{eqA2}
\int\limits_{r_1}^{r_2}\eta(r)\, dr\geqslant 1\,.
\end{equation}
Given a function $Q$ measurable with a respect to a measure
$\widetilde{v},$ we set
\begin{equation}\label{eq26}
q_{p_0}(r)=\frac{1}{r}\int\limits_{\widetilde{S}(p_0,
r)}Q(p)\,ds_{\widetilde{h_*}}(p)\,.
\end{equation}

\medskip
The following theorem holds.

\begin{theorem}\label{th1}
{\,\sl Let $D, D_*$ be domains in ${\Bbb S}$ and ${\Bbb S}_*,$
respectively, having compact closures $\overline{D}\subset {\Bbb S}$
and $\overline{D_*}\subset {\Bbb S}_*,$ in addition, $\partial D$
and $\partial D_*$ have a finite number of components. Let $Q:{\Bbb
S}_*\rightarrow (0, \infty)$ be a given function measurable with
respect to the measure $\widetilde{v_*}$ on ${\Bbb S}_*$,
$Q(p)\equiv 0$ in ${\Bbb S}_*\setminus f(D).$ Let $f:D\rightarrow
D_*$ be an open, discrete and closed mapping of $D$ onto $D_*$
satisfying the relation~(\ref{eq2*A}). If, for each point $p_0\in
\partial D_*$ there is $r_0=r_0(p_0)>0$ such that $q_{p_0}(r)<\infty$
for almost every $r\in (0, r_0),$ then $f$ has a continuous
extension $f:\overline{D}_P\rightarrow \overline{D_*}_P,$
$f(\overline{D}_P)=\overline{D_*}_P.$ }
\end{theorem}

\medskip
The boundary of the domain $D$ is called {\it weakly flat} at the
point $x_0\in \partial D, $ if for any $P> 0$ and for any
neighborhood $U$ of a point $x_0 $ there is a neighborhood $V\subset
U$ of the same point such that $M(\Gamma(E, F, D))> P$ for any
continua $E, F \subset D,$ which intersect $\partial U$ and
$\partial V.$ The boundary of the domain $D$ is called weakly flat
if the corresponding property is fulfilled at any point of the
boundary $D.$ In the case of a good boundary of the domain $D,$ we
have the following version of Theorem~\ref{th1}.

\begin{theorem}\label{th3}
{\,\sl Let $D, D_*$ be domains in ${\Bbb S}$ and ${\Bbb S}_*,$
respectively, having compact closures $\overline{D}\subset {\Bbb S}$
and $\overline{D_*}\subset {\Bbb S}_*,$ in addition, $\partial D$
has a weakly flat boundary, and $\partial D_*$ have a finite number
of components. Let $Q:{\Bbb S}_*\rightarrow (0, \infty)$ be a given
function measurable with respect to the measure $\widetilde{v_*}$ on
${\Bbb S}_*$, $Q(p)\equiv 0$ in ${\Bbb S}_*\setminus f(D).$ Let
$f:D\rightarrow D_*$ be an open, discrete and closed mapping of $D$
onto $D_*$ satisfying the relation~(\ref{eq2*A}). If, for each point
$p_0\in \partial D_*$ there is $r_0=r_0(p_0)>0$ such that
$q_{p_0}(r)<\infty$ for almost every $r\in (0, r_0),$ then $f$ has a
continuous extension $f:\overline{D}\rightarrow \overline{D_*}_P,$
$f(\overline{D})=\overline{D_*}_P.$ }
\end{theorem}

\medskip
A result similar to the assertion of Theorem~\ref{th1} was
previously obtained for homeomorphisms in~\cite[Theorem~2]{RV$_2$}.
Also, under some more general assumptions, we obtained this
assertion on Riemannian manifolds, see~\cite[Theorem~1.4]{Sev$_3$}.

\section{Preliminaries}

Given a mapping $f:D\rightarrow {\Bbb S}_*$ and a set $E\subset
\overline{D}\subset{\Bbb S},$ we put
$$C(f, E)=\{y\in {\Bbb S}_*:\exists\, x\in E, x_k\in D: x_k\rightarrow x, f(x_k)
\rightarrow y, k\rightarrow\infty\}\,.$$
The following statement holds.

\begin{proposition}\label{pr1}
{\,\sl Assume that, a domain $D\subset {\Bbb S}$ has a finite number
of boundary components $\Gamma_1, \Gamma_2,\ldots, \Gamma_n\subset
\partial D.$ Then:

\medskip
1) for any $\Gamma_i,$ $i=1,2,\ldots, n$ there is a neighborhood
$U_i\subset {\Bbb S}$ and a conformal mapping $H$ of $U^*_i:=U_i\cap
D$ onto $R=\{z\in {\Bbb C}: 0\leqslant r_i<|z|<1\}$ such that
$\gamma_i:=\partial U_i^{\,*}\cap D$ is a closed Jordan path
$$C(H, \gamma_i) = \{z\in {\Bbb C}: |z| = 1\};\quad  C(H, \Gamma_i) =
\{z\in {\Bbb C}: |z| = r_i\}\,,$$
while $r_i=0$ if and only if $\Gamma$ degenerates into a point.
Moreover, $H$ extends to a homeomorphism of $\overline{U_i^{\,*}}_P$
onto $\overline{R},$ see \cite[Lemma~2]{RV$_2$};

\medskip
2) a space $\overline{D}_P$ is metrizable with some metric
$\rho:\overline{D}_P\times\overline{D}_P\rightarrow {\Bbb R}$ such
that, the convergence of any sequence $x_n\in D,$ $n=1,2,\ldots ,$
to some prime end $P\in E_D$ is equivalent to the convergence $x_n$
in one of spaces $\overline{U_i^{\,*}}_P,$ see
\cite[Remark~2]{RV$_2$};

\medskip
3) any prime end $P\in E_D$ contains a chain of cuts $\sigma_m,$
$m=1,2,\ldots,$ which belong to spheres $\widetilde{S}(z_0, r_m),$
$r_m\rightarrow 0$ as $m\rightarrow\infty,$
see~\cite[Remark~1]{RV$_2$};

\medskip
4) for any $P\subset E_D$ its impression $I(P)$ is a continuum in
$\partial D,$  while there is some unique $1\leqslant i\leqslant n$
such that $I(P)\subset \Gamma_i,$ see \cite[Proposition~1,
Remark~1]{RV$_2$}.}

\end{proposition}

\medskip Let $X$ and $X^{\,\prime}$ be metric spaces with metrics $d$ and $d^{\,\prime},$
respectively, and let $D$ be a domain in $X.$ Let $f:D\rightarrow
X^{\,\prime}$ be a discrete open mapping, let $\beta:
[a,\,b)\rightarrow X^{\,\prime}$ be a path, and let
$x\in\,f^{-1}\left(\beta(a)\right).$ A path $\alpha:
[a,\,c)\rightarrow D$ is called a {\it maximal $f$-lifting} of
$\beta$ starting at $x,$ if $(1)\quad \alpha(a)=x\,;$ $(2)\quad
f\circ\alpha=\beta|_{[a,\,c)};$ $(3)$\quad for
$c<c^{\prime}\leqslant b,$ there is no a path $\alpha^{\prime}:
[a,\,c^{\prime})\rightarrow D$ such that
$\alpha=\alpha^{\prime}|_{[a,\,c)}$ and $f\circ
\alpha^{\,\prime}=\beta|_{[a,\,c^{\prime})}.$ If $X$ and
$X^{\,\prime}$ are locally compact, $X$ is locally connected, and
$f:D \rightarrow X^{\,\prime}$ is discrete and open, then there is a
maximal $f$-lifting of $\beta$ starting at $x,$
see~\cite[Lemma~2.1]{SM}. The following statement is proved
in~\cite[Lemma~2.1]{Sev$_3$} (for the space ${\Bbb R}^n$ see, for
example, \cite[Theorem~3.7]{Vu}).

\medskip
\begin{lemma}\label{lem9}
{\sl Let $X$ and $X^{\,\prime}$ be metric spaces, let $X$ be locally
connected, let $X^{\,\prime}$ be locally compact, let $D$ be a
domain in $X,$ and let $f:D \rightarrow X^{\,\prime}$ be a discrete
open and closed mapping of $D$ onto $D^{\,\prime}\subset
X^{\,\prime}.$ Assume that $\overline{D}$ is compact. If $\beta:
[a,\,b)\rightarrow X^{\,\prime}$ is a path such that
$x\in\,f^{\,-1}(\beta(a)),$ then there is a whole $f$-lifting of
$\beta$ starting at $x,$ in other words, there is a path $\alpha:
[a,\,b)\rightarrow X$ such that $f(\alpha(t))=\beta(t)$ for any
$t\in [a,\,b).$ Moreover, if $\beta(t)$ has a limit
$\lim\limits_{t\rightarrow b-0}\beta(t):=B_0\in D^{\,\prime},$ then
$\alpha$ has a continuous extension to $b$ and $f(\alpha(b))=B_0.$ }
\end{lemma}

The following analogue of Fubini's theorem for Riemannian surfaces
holds (see~\cite{Sev$_6$}).

\begin{lemma}\label{lem2}
{\sl Let $U$ be some normal neighborhood of the point $p_0\in {\Bbb
S},$ and let $Q:U\rightarrow [0, \infty]$ be a function measurable
with respect to the measure $\widetilde{h},$ and $d_0:={\rm
dist}\,(p_0,
\partial U):=\inf\limits_{p\in
\partial U}\widetilde{h}(p_0, p).$ Then, for any $0<r_0\leqslant d_0$
\begin{equation}\label{eq6}
\int\limits_{\widetilde{B}(p_0,
r_0)}Q(p)\,d\widetilde{v}(p)=\int\limits_0^{r_0}\int\limits_{\widetilde{S}(p_0,
r)}Q(p)\,ds_{\widetilde{h}}(p)\,dr\,,
\end{equation}
where $d\widetilde{v}(p)$ and $ds_{\widetilde{h}}$ are area and
length elements on ${\Bbb S},$ respectively, see~(\ref{eq1}), and
the disk $\widetilde{B}(p_0, r_0)$ and the circle
$\widetilde{S}(p_0, r)$ are defined in~(\ref{eq4}).}
\end{lemma}

\section{Proof of Theorem~\ref{th1}}

\medskip
{\it Proof of Theorem~\ref{th1}.} Let $P\in E_D$ and let $\Gamma$ be
a component of $\partial D$ such that $I(P)\subset \Gamma$ (it
exists by item~4) of Proposition~\ref{pr1}). Let $U\subset {\Bbb S}$
be a neighborhood $\Gamma$ which corresponds to
Proposition~\ref{pr1}, and let $H$ be a corresponding conformal
mapping of a domain $U^{\,*}:=U\cap D$ onto the ring $R=\{z\in {\Bbb
C}: 0<r<|z|<1\}$ such that $\gamma:=\partial U^{\,*}\cap D$ is a
closed Jordan path,
$$C(H, \gamma) = \{z\in {\Bbb C}: |z| = 1\};\quad  C(H, \Gamma) =
\{z\in {\Bbb C}: |z| = r\}\,.$$
%
%
%By Proposition~\ref{pr1}, there is a chain of cuts $\sigma_n,$
%corresponding to a prime end $P,$ which belongs to circles
%$\widetilde{S}(p_0, r_n),$ $p_0\in
%\partial D,$ $r_n\rightarrow 0$ as $n\rightarrow\infty.$
Let $d_n,$ $n=1,2,\ldots ,$ be a sequence of domains corresponding
to $P.$ It follows from the definition of $U^{\,*}$ that, all $d_n$
belongs to $U^{\,*}$ for sufficiently large $n\in {\Bbb N}.$ We may
assume that $d_1\subset U^*.$ Now, we set $$g:=f\circ H^{\,-1}\,.$$
Since $H$ is conformal, it preserves the modulus of families of
paths (see, e.g., \cite[Corollary~1.2]{Sev$_4$}). Thus, $g$
satisfies the relation~(\ref{eq2*A}), as well. Thus, for the proof
of Theorem~\ref{th1}, it is sufficient to establish the continuous
extension $\overline{g}:\overline{R}\rightarrow \overline{D_*}_P.$

\medskip
Let $x_0\in \partial R.$ Assume that the conclusion about the
continuous extension of the mapping $g$ to the point $x_0$ is not
correct. Since $\overline{D^{\,*}}$ is compact, due to item~2) of
Proposition~\ref{pr1} the space $\overline{D_*}_P$ is compact, as
well. Then there are sequences $x_k, y_k\in R,$ $k=1,2,\ldots,$ and
$P_1, P_2\in \overline{D_*}_P,$ $P_1\ne P_2,$ such that $x_k,
y_k\rightarrow x_0 $ as $k\rightarrow \infty,$ and $\rho(g(x_k),
P_1)\rightarrow 0,$ $\rho(g(y_k), P_2)\rightarrow 0$ as
$k\rightarrow\infty,$ where $\rho$ is some metric in
$\overline{D_*}_P$ (see item~2) in Proposition~\ref{pr1}). Since $f$
is closed, $f$ is boundary preserving (see
e.g.~\cite[Proposition~2.1]{Sev$_7$}). Thus, $P_1, P_2\in E_{D_*}.$

\medskip
Let $E_k, G_k$ $k=1,2,\ldots, $ be sequences of domains
corresponding to prime ends $P_1, P_2,$ respectively. Let us show
that there exists $k_0\in {\Bbb N}$ such that
\begin{equation}\label{eq12}
E_k\cap G_k=\varnothing\quad \forall\,\,k\geqslant k_0\,.
\end{equation}
Suppose the contrary, i.e., suppose that for every $l=1,2,\ldots$
there exists an increasing sequence $k_l,$ $l=1,2,\ldots,$ such that
$p_{k_l}\in E_{k_l}\cap G_{k_l},$ $l=1,2,\ldots .$ Now
$p_{k_l}\rightarrow P_1$ and $p_{k_l}\rightarrow P_2,$
$l\rightarrow\infty.$ Let $\rho$ be the metric on $\overline{D_*}_P$
defined in Proposition~\ref{pr1}. By the triangle inequality,
$$\rho(P_1, P_2)\leqslant \rho(P_1, x_{k_l})+\rho(x_{k_l}, P_2)
\rightarrow 0,\qquad l\rightarrow\infty\,,$$
that contradicts to the condition $P_1\ne P_2.$

\medskip
Due to Proposition~\ref{pr1}, we may consider that a chain of cuts
$\sigma_n,$ which corresponds to domains $E_k$ and the prime end
$P_1$ belongs to circles $\widetilde{S}(p_0, r_n),$ $p_0\in
\partial D_*,$ $r_n\rightarrow 0$ as
$n\rightarrow\infty.$ Since $\rho(g(x_k), P_1)\rightarrow 0,$ there
is a number $k_1\in {\Bbb N}$ such that $g(x_k)\in E_1$ for any
$k\geqslant k_1.$ Similarly, there is a number $k_2\in {\Bbb N}$
such that $g(x_k)\in E_2$ for any $k\geqslant k_2.$ Join points
$g(x_k)$ and $g(x_{k_2})$ inside $E_2$ by a path $C_k:[0,
1]\rightarrow E_2$ such that $C_k(0)=g(x_k),$ $C_k(1)=g(x_{k_2}).$

\medskip
Similarly, there is a number $k_3\in {\Bbb N}$ such that $g(y_k)\in
G_1$ for any $k\geqslant k_3.$ Join points $g(y_k)$ and $g(y_{k_3})$
inside $G_1$ by a path $\widetilde{C_k}:[0, 1]\rightarrow G_1$ such
that $\widetilde{C_k}(0)=g(y_k),$ $\widetilde{C_k}(1)=g(y_{k_3}),$
see Figure~\ref{fig1} on this occasion.
\begin{figure}[h]
\centerline{\includegraphics[scale=0.4]{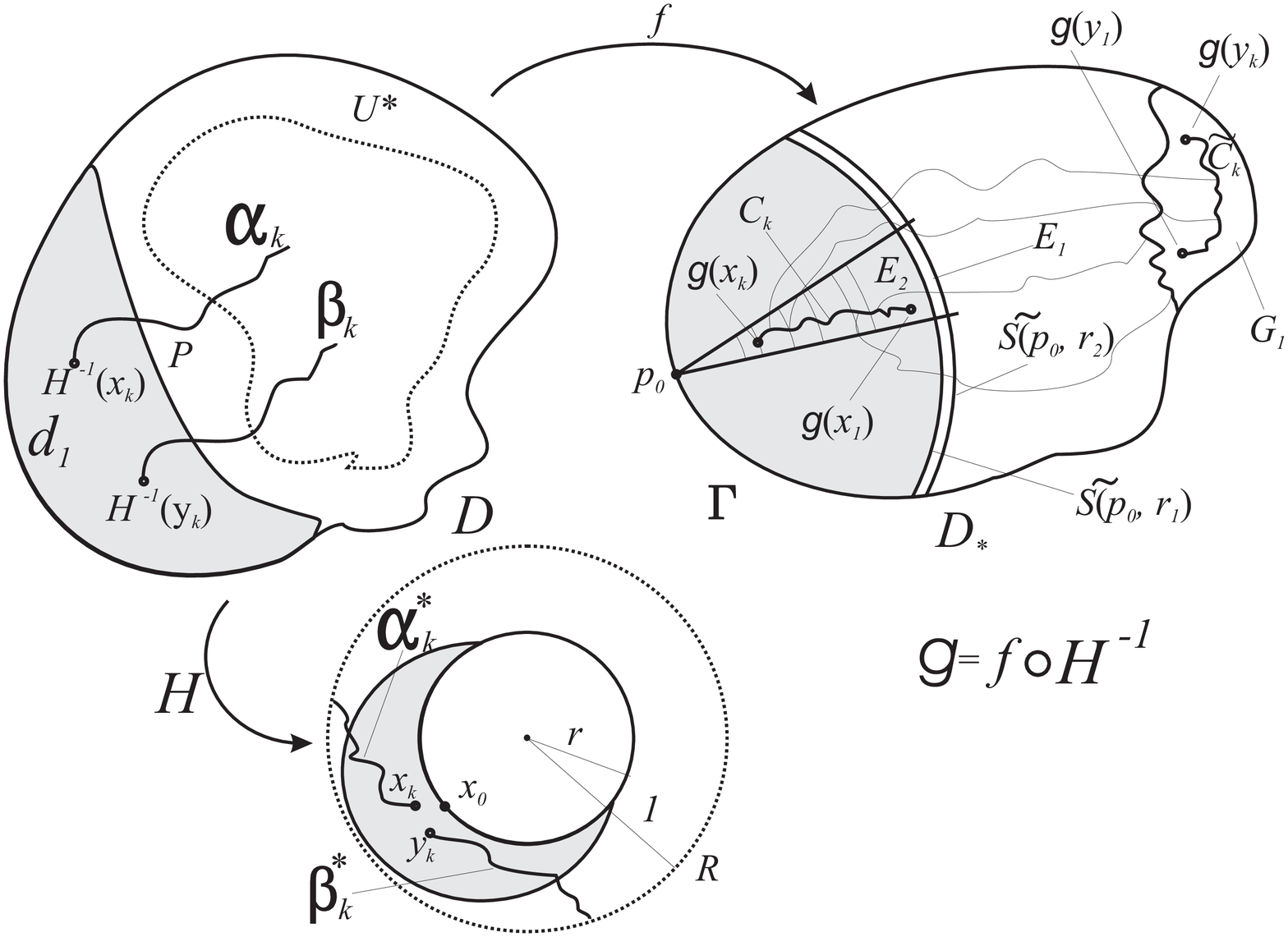}} \caption{To
the proof of Theorem~\ref{th1}}\label{fig1}
\end{figure}

\medskip
Given a path $\gamma:[a, b]\rightarrow {\Bbb S},$ we set
$$|\gamma|:=\{p\in {\Bbb S}: \exists\,t\in[a, b]:
\gamma(t)=p\}\,.$$
Let us to show that
\begin{equation}\label{eq11}
\Gamma(|C_k|, \widetilde{|C_k}|, D_*)>\Gamma(\widetilde{S}(p_0,
r_1), \widetilde{S}(p_0, r_2), \widetilde{A}(p_0, r_1, r_2))\,,
\end{equation}
where $\widetilde{A}(\widetilde{p_0}, r_1, r_2)$ is defined
in~(\ref{eq1**}).
Assume that $\gamma\in \Gamma(C_k, \widetilde{C_k}, D_*),$
$\gamma:[0, 1]\rightarrow D_*.$ By the construction, $|\gamma|\cap
E_2\ne\varnothing\ne |\gamma|\cap (D_*\setminus E_2).$ Thus,
\begin{equation}\label{eq13}
|\gamma|\cap \partial E_2\ne\varnothing
\end{equation} (see \cite[Theorem 1, $\S\,$46, item I]{Ku}).
Thus, there is $t_1\in [0, 1]$ such that $\gamma(t_1)\in \partial
E_2.$ Since $E_2\subset E_1,$ we have that $\overline{E_2}\subset
\overline{E_1}.$ Thus, either $\gamma(t_1)\in E_1,$ or
$\gamma(t_1)\in \partial E_1\cap D_*\subset
\widetilde{S}(\widetilde{p_0}, r_1).$ The second case is impossible,
because  Observe that, $t_2<1$ because $\gamma(t_1)\in
\partial E_2\cap D_*\subset\widetilde{S}(\widetilde{p_0},
r_2).$ Now, $\gamma(t_1)\in E_1.$ By the condition~(\ref{eq12}),
$\gamma(t_1)\not\in G_1.$ Now, $t_1\ne 1.$ Set
$\gamma_1:=\gamma|_{[t_1, 1]}.$ Without loss of generality, we may
assume that $\gamma(t)\not\in \partial E_2$ for $t\in (t_1, 1].$

\medskip
Arguing similarly, we obtain that $|\gamma_1|\cap
E_1\ne\varnothing\ne |\gamma|\cap (D_*\setminus 12).$ Thus,
\begin{equation}\label{eq13A}
|\gamma_1|\cap \partial E_1\ne\varnothing
\end{equation}
(see \cite[Theorem 1, $\S\,$46, item I]{Ku}). Now, there is $t_2\in
(t_1, 1)$ such that $\gamma(t_2)\in \partial E_1.$ Without loss of
generality, we may assume that $\gamma(t)\in E_1$ for any $t\in
[t_1, t_2].$ Set $\gamma_2:=\gamma_1|_{[t_1, t_2]}.$ Now $\gamma_2$
is a subpath of $\gamma$ such that $\gamma(t_1)\in \partial E_1\cap
D_*\subset \widetilde{S}(\widetilde{p_0}, r_1)$ and $\gamma(t_2)\in
\partial E_2\cap D_*\subset \widetilde{S}(p_0,
r_2),$ while $\gamma(t)\in \widetilde{A}(p_0, r_1, r_2),$ that
proves~(\ref{eq11}).

\medskip
By Lemma~\ref{lem9} paths $C_k$ and $\widetilde{C_k}$ have whole
$f$-liftings $\alpha_k$ and $\beta_k$ in $D$ starting at points
$H^{\,-1}(x_k)$ and $H^{\,-1}(y_k),$ respectively. Observe that
points $f(H^{\,-1}(x_1))$ and $f(H^{\,-1}(y_1))$ cannot have more
than a finite number of pre-images in the domain $D$ under the
mapping $f.$ Indeed, by Proposition~2.1 in~\cite{Sev$_7$}, $f$ is
proper, so that $f^{\,-1}(f(H^{\,-1}(x_1)))$ is a compact set in
$D.$ If there is $z_m\in f^{\,-1}(f(H^{\,-1}(x_1))),$
$m=1,2,\ldots,$ $z_n\ne z_k$ for $n\ne k,$ then we may find a
subsequence $z_{m_p},$ $p=1,2,\ldots ,$ such that
$z_{m_p}\rightarrow z_0$ as $p\rightarrow \infty$ for some $z_0\in
D.$ The latter contradicts the discreteness of $f.$ Thus,
$f^{\,-1}(f(H^{\,-1}(x_1)))$ is finite. Similarly,
$f^{\,-1}(f(H^{\,-1}(y_1)))$ is finite. Thus, there is $K_0\in {\Bbb
N},$ $K_0>\max\{k_0, k_1, k_2\},$ and $\delta_0>0$ such that $${\rm
dist}(\alpha_k(1), \partial D)={\rm
dist}(f^{\,-1}(f(H^{\,-1}(x_1))), \partial D)\geqslant \delta_0$$
and $${\rm dist}(\beta_k(1), \partial D)={\rm
dist}(f^{\,-1}(f(H^{\,-1}(y_1))), \partial D)\geqslant \delta_0$$
for $k\geqslant K_0.$ Now we set
$$t_k:=\sup\{t\in [0, 1]: \alpha_k(t)\in U^{\,*}\}\,,\quad
p_k:=\sup\{t\in [0, 1]: \beta_k(t)\in U^{\,*}\}\,.$$
By the proving above and by the definition of the mapping $H,$ we
obtain that
\begin{equation}\label{eq1}
{\rm dist\,}(H(\alpha_k(t_k)), x_0)\geqslant r_0\,,\quad {\rm
dist\,}(H(\beta_k(t_k)), x_0)\geqslant r_0\,,\quad k\geqslant K_0
\end{equation}
for some $r_0>0.$ Set
$$\alpha^{\,*}_k:=H\left(\alpha_k|_{[0, t_k]}\right)\,,\qquad
\beta^{\,*}_k:=H\left(\beta_k|_{[0, p_k]}\right)\,.$$
Since $R$ is a union of two circles which are $C^1$-manifolds, $R$
has a weakly flat boundary (see e.g.~\cite[Theorems~17.10,
17.12]{Va}). Thus, for any $P>0$ there is $k=k_P\geqslant K_0$ such
that
\begin{equation}\label{eq7}
M(\Gamma(|\alpha^{\,*}_k|, |\beta^{\,*}_k|,
R))>P\qquad\forall\,\,k\geqslant k_P\,.
\end{equation}

\medskip
We show that the condition~(\ref{eq7}) contradicts~(\ref{eq2*A}). By
the construction,
\begin{equation}\label{eq2}
g(\alpha^{\,*}_k)\subset C_k, \qquad g(\beta^{\,*}_k)\subset
\widetilde{C_k}\,.
\end{equation}
By~(\ref{eq11}) and~(\ref{eq2}),
$$g(\Gamma(|\alpha^{\,*}_k|, |\beta^{\,*}_k|,
R))\subset \Gamma(|C_k|, \widetilde{|C_k}|,
D_*)>\Gamma(\widetilde{S}(\widetilde{p_0}, r_1), \widetilde{S}(p_0,
r_2), \widetilde{A}(p_0, r_1, r_2))\,.$$
Thus,
$$\Gamma(|\alpha^{\,*}_k|, |\beta^{\,*}_k|,
R)>\Gamma_g(p_0, r_1, r_2)\,.$$
By the definition of a mapping $f$ in~(\ref{eq2*A}) and due to
remarks made in the begin of the proof, we obtain that
\begin{equation}\label{eq3}
M(\Gamma(|\alpha^{\,*}_k|, |\beta^{\,*}_k|, R))\leqslant
M(\Gamma_g(p_0, r_1, r_2))\leqslant \int\limits_{D_*\cap
\widetilde{A}(p_0, r_1, r_2)} Q(p)\cdot \eta^{2}(\widetilde{h_*}(p,
p_0))\, d\widetilde{v_*}(p)\,.
\end{equation}
Set $\widetilde{Q}(p)=\max\{Q(p), 1\}$ and
$$\widetilde{q}_{p_0}(r)=\int\limits_{S(p_0,
r)}\widetilde{Q}(p)\,ds_{\widetilde{h_*}}(p)\,.$$ Now, we have also
that $\widetilde{q}_{p_0}(r)\ne \infty$ for any $r\in [r_1,r_2].$
Set
$$I=I(p_0,r_1,r_2)=\int\limits_{r_1}^{r_2}\
\frac{dr}{r\widetilde{q}_{p_0}(r)}\,.$$
Observe that $I\ne 0,$ because $\widetilde{q}_{p_0}(r)\ne \infty$
for any $r\in [r_1, r_2].$ Besides that, note that $I\ne\infty,$
since, for some constant $C>0,$
$$I\leqslant C\cdot\log\frac{r_2}{r_1}<\infty\,,\quad i=1,2, \ldots, p\,.$$
Now, we put
$$\eta(r)=\begin{cases}
\frac{1}{Ir\widetilde{q}_{p_0}(r)}\,,&
r\in [r_1,r_2]\,,\\
0,& r\not\in [r_1,r_2]\,.
\end{cases}$$
Observe that, a function~$\eta$ satisfies the
condition~$\int\limits_{r_1}^{r_2}\eta(r)\,dr=1,$ therefore it can
be substituted into the right side of the inequality~(\ref{eq3})
with the corresponding values $f,$ $r_1$ and $r_2.$

On the other hand, by Lemma~\ref{lem2},
$$\int\limits_{D_*\cap \widetilde{A}(p_0, r_1,
r_2)}Q(p)\cdot \eta^{2}(\widetilde{h_*}(p, p_0))\,
d\widetilde{v_*}(p)=$$
\begin{equation}\label{eq7C}
= \int\limits_{r_1}^{r_2}\int\limits_{S(p_0,
r)}Q(y)\eta^2(\widetilde{h_*}(p,
p_0))\,ds_{\widetilde{h_*}}(p)\,dr\,\leqslant
\end{equation}
$$\leqslant\frac{C}{I^2}\int\limits_{r_1}^{r_2}r
\widetilde{q}_{p_0}(r)\cdot
\frac{dr}{r^2\widetilde{q}_{p_0}(r)}=\frac{C}{I}<\infty\,,$$
where $C$ is some constant. Due to~(\ref{eq3}) and~(\ref{eq7C}), we
obtain that
\begin{equation}\label{eq7B}
M(\Gamma(|\alpha^{\,*}_k|, |\beta^{\,*}_k|, R))\leqslant
\frac{C}{I}<\infty\end{equation}
for sufficiently large $k\in {\Bbb N}.$ The relation~(\ref{eq7B})
contradicts the condition~(\ref{eq7}). The contradiction obtained
above refutes the assumption that there is no limit of the mapping
$f$ at the point $x_0.$

It remains to show that
$\overline{f}(\overline{D}_P)=\overline{D_*}_P.$ Obviously
$\overline{f}(\overline{D}_P)\subset\overline{D_*}_P.$ Let us to
show that $\overline{D_*}_P\subset \overline{f}(\overline{D}_P).$
Indeed, let $y_0\in \overline{D_*}_P.$ Then either $y_0\in D_*,$ or
$y_0\in E_{D_*}.$ If $y_0\in D_*,$ then $y_0=f(x_0)$ and $y_0\in
\overline{f}(\overline{D}),$ because $f$ maps $D$ onto $D_*.$
Finally, let $y_0\in E_{D_*}.$ Due to Proposition~\ref{pr1}, there
is a sequence $y_k\in D_*$ such that $\rho(y_k, y_0)\rightarrow 0$
as $k\rightarrow\infty,$ $y_k=f(x_k)$ and $x_k\in D,$ where $\rho$
is defined in Proposition~\ref{pr1}. Since $\overline{D}$ is
compact, we may assume that $x_k\rightarrow x_0,$ where
$x_0\in\overline{D}_P.$ Note that $x_0\in E_D,$ because $f$ is open.
Thus $f(x_0)=y_0\in \overline{f}(E_D)\subset
\overline{f}(\overline{D}_P).$ Theorem is completely proved.~$\Box$

\medskip
By Lemma~\ref{lem2} we immediately obtain the following consequence
from Theorem~\ref{th1}.

\medskip
\begin{corollary}\label{cor1}{\sl
The statement of Theorem~\ref{th1} remains true if, in this theorem,
we replace the corresponding condition on $Q$ by the condition:
$Q\in L^1(D_*).$}
\end{corollary}

\medskip
{\it Proof of Theorem~\ref{th3}.} Since the proof is very similar to
the proof of Theorem~\ref{th1}, we restrict us by the sketch of the
proof. Let $x_0\in \partial D.$ Assume that the conclusion about the
continuous extension of the mapping $g$ to the point $x_0$ is not
correct. Since $\overline{D^{\,*}}$ is compact, due to item~2) of
Proposition~\ref{pr1} the space $\overline{D_*}_P$ is compact, as
well. Then there are sequences $x_k, y_k\in R,$ $k=1,2,\ldots,$ and
$P_1, P_2\in \overline{D_*}_P,$ $P_1\ne P_2,$ such that $x_k,
y_k\rightarrow x_0 $ as $k\rightarrow \infty,$ and $\rho(g(x_k),
P_1)\rightarrow 0,$ $\rho(g(y_k), P_2)\rightarrow 0$ as
$k\rightarrow\infty,$ where $\rho$ is some metric in
$\overline{D_*}_P$ (see item~2) in Proposition~\ref{pr1}). Since $f$
is closed, $f$ is boundary preserving (see
e.g.~\cite[Proposition~2.1]{Sev$_7$}). Thus, $P_1, P_2\in E_{D_*}.$

\medskip
Let $E_k, G_k$ $k=1,2,\ldots, $ be sequences of domains
corresponding to prime ends $P_1, P_2,$ respectively. Similarly to
the relation~(\ref{eq12}), we may show that
\begin{equation}\label{eq12A}
E_k\cap G_k=\varnothing\quad \forall\,\,k\geqslant k_0\,.
\end{equation}

\medskip
Due to Proposition~\ref{pr1}, we may consider that a chain of cuts
$\sigma_n,$ which corresponds to domains $E_k$ and the prime end
$P_1$ belongs to circles $\widetilde{S}(p_0, r_n),$ $p_0\in
\partial D_*,$ $r_n\rightarrow 0$ as
$n\rightarrow\infty.$ Since $\rho(g(x_k), P_1)\rightarrow 0,$ there
is a number $k_1\in {\Bbb N}$ such that $g(x_k)\in E_1$ for any
$k\geqslant k_1.$ Similarly, there is a number $k_2\in {\Bbb N}$
such that $g(x_k)\in E_2$ for any $k\geqslant k_2.$ Join points
$g(x_k)$ and $g(x_{k_2})$ inside $E_2$ by a path $C_k:[0,
1]\rightarrow E_2$ such that $C_k(0)=g(x_k),$ $C_k(1)=g(x_{k_2}).$

\medskip
Similarly, there is a number $k_3\in {\Bbb N}$ such that $g(y_k)\in
G_1$ for any $k\geqslant k_3.$ Join points $g(y_k)$ and $g(y_{k_3})$
inside $G_1$ by a path $\widetilde{C_k}:[0, 1]\rightarrow G_1$ such
that $\widetilde{C_k}(0)=g(y_k),$ $\widetilde{C_k}(1)=g(y_{k_3}).$

Similarly to the relation~(\ref{eq11}), we prove that
 show that
\begin{equation}\label{eq11A}
\Gamma(|C_k|, \widetilde{|C_k}|, D_*)>\Gamma(\widetilde{S}(p_0,
r_1), \widetilde{S}(p_0, r_2), \widetilde{A}(p_0, r_1, r_2))\,,
\end{equation}
where $\widetilde{A}(\widetilde{p_0}, r_1, r_2)$ is defined
in~(\ref{eq1**}).

\medskip
By Lemma~\ref{lem9} paths $C_k$ and $\widetilde{C_k}$ have whole
$f$-liftings $\alpha_k$ and $\beta_k$ in $D$ starting at points
$x_k$ and $y_k,$ respectively. Observe that points $f(x_1)$ and
$f(y_1)$ cannot have more than a finite number of pre-images in the
domain $D$ under the mapping $f.$ Indeed, by Proposition~2.1
in~\cite{Sev$_7$}, $f$ is proper, so that $f^{\,-1}(f(x_1))$ is a
compact set in $D.$ If there is $z_m\in f^{\,-1}(f(x_1)),$
$m=1,2,\ldots,$ $z_n\ne z_k$ for $n\ne k,$ then we may find a
subsequence $z_{m_p},$ $p=1,2,\ldots ,$ such that
$z_{m_p}\rightarrow z_0$ as $p\rightarrow \infty$ for some $z_0\in
D.$ The latter contradicts the discreteness of $f.$ Thus,
$f^{\,-1}(f(x_1))$ is finite. Similarly, $f^{\,-1}(f(y_1))$ is
finite. Thus, there is $K_0\in {\Bbb N},$ $K_0>\max\{k_0, k_1,
k_2\},$ and $\delta_0>0$ such that
\begin{equation}\label{eq1C}
{\rm dist}(\alpha_k(1),
\partial D)\geqslant \delta_0
\end{equation}
and
\begin{equation}\label{eq1B}
{\rm dist}(\beta_k(1), \partial D)\geqslant \delta_0
\end{equation}
for $k\geqslant K_0.$
Thus, due to the conditions~(\ref{eq1C}) and~(\ref{eq1B}), for any
$P>0$ there is $k=k_P\geqslant K_0$ such that
\begin{equation}\label{eq7A}
M_{\alpha}(\Gamma(|\alpha_k|, |\beta_k|,
D))>P\qquad\forall\,\,k\geqslant k_P\,.
\end{equation}

\medskip
We show that the condition~(\ref{eq7A}) contradicts~(\ref{eq2*A}).
By~(\ref{eq11A}),
$$f(\Gamma(|\alpha_k|, |\beta_k|,
D))=\Gamma(|C_k|, \widetilde{|C_k}|,
D_*)>\Gamma(\widetilde{S}(\widetilde{p_0}, r_1), \widetilde{S}(p_0,
r_2), \widetilde{A}(p_0, r_1, r_2))\,.$$
Thus,
$$\Gamma(|\alpha_k|, |\beta_k|,
D)>\Gamma_f(p_0, r_1, r_2)\,.$$
By the definition of a mapping $f$ in~(\ref{eq2*A}), we obtain that
\begin{equation}\label{eq3B}
M(\Gamma(|\alpha_k|, |\beta_k|, D))\leqslant M(\Gamma_f(p_0, r_1,
r_2))\leqslant \int\limits_{D_*\cap \widetilde{A}(p_0, r_1, r_2)}
Q(p)\cdot \eta^{2}(\widetilde{h_*}(p, p_0))\, d\widetilde{v_*}(p)\,.
\end{equation}
Set $\widetilde{Q}(p)=\max\{Q(p), 1\}$ and
$$\widetilde{q}_{p_0}(r)=\int\limits_{S(p_0,
r)}\widetilde{Q}(p)\,ds_{\widetilde{h_*}}(p)\,.$$ Now, we have also
that $\widetilde{q}_{p_0}(r)\ne \infty$ for any $r\in [r_1,r_2].$
Set
$$I=I(p_0,r_1,r_2)=\int\limits_{r_1}^{r_2}\
\frac{dr}{r\widetilde{q}_{p_0}(r)}\,.$$
Observe that $I\ne 0,$ because $\widetilde{q}_{p_0}(r)\ne \infty$
for any $r\in [r_1, r_2].$ Besides that, note that $I\ne\infty,$
since, for some constant $C>0,$
$$I\leqslant C\cdot\log\frac{r_2}{r_1}<\infty\,,\quad i=1,2, \ldots, p\,.$$
Now, we put
$$\eta(r)=\begin{cases}
\frac{1}{Ir\widetilde{q}_{p_0}(r)}\,,&
r\in [r_1,r_2]\,,\\
0,& r\not\in [r_1,r_2]\,.
\end{cases}$$
Observe that, a function~$\eta$ satisfies the
condition~$\int\limits_{r_1}^{r_2}\eta(r)\,dr=1,$ therefore it can
be substituted into the right side of the inequality~(\ref{eq3B})
with the corresponding values $f,$ $r_1$ and $r_2.$

On the other hand, by Lemma~\ref{lem2},
$$\int\limits_{D_*\cap \widetilde{A}(p_0, r_1,
r_2)}Q(p)\cdot \eta^{2}(\widetilde{h_*}(p, p_0))\,
d\widetilde{v_*}(p)=$$
\begin{equation}\label{eq7D}
= \int\limits_{r_1}^{r_2}\int\limits_{S(p_0,
r)}Q(y)\eta^2(\widetilde{h_*}(p,
p_0))\,ds_{\widetilde{h_*}}(p)\,dr\,\leqslant
\end{equation}
$$\leqslant\frac{C}{I^2}\int\limits_{r_1}^{r_2}r
\widetilde{q}_{p_0}(r)\cdot
\frac{dr}{r^2\widetilde{q}_{p_0}(r)}=\frac{C}{I}<\infty\,,$$
where $C$ is some constant. Due to~(\ref{eq3B}) and~(\ref{eq7C}), we
obtain that
\begin{equation}\label{eq7E}
M(\Gamma(|\alpha_k|, |\beta_k|, D))\leqslant
\frac{C}{I}<\infty\end{equation}
for sufficiently large $k\in {\Bbb N}.$ The relation~(\ref{eq7E})
contradicts the condition~(\ref{eq7A}). The contradiction obtained
above refutes the assumption that there is no limit of the mapping
$f$ at the point $x_0.$

It remains to show that
$\overline{f}(\overline{D})=\overline{D_*}_P.$ This fact may be
established similarly to the last part of the proof of
Theorem~\ref{th1}.~$\Box$

\medskip
By Lemma~\ref{lem2} we immediately obtain the following consequence
from Theorem~\ref{th3}.

\medskip
\begin{corollary}\label{cor2}{\sl
The statement of Theorem~\ref{th3} remains true if, in this theorem,
we replace the corresponding condition on $Q$ by the condition:
$Q\in L^1(D_*).$}
\end{corollary}

\begin{remark}\label{rem2}
The statements of Theorems~\ref{th1} and~\ref{th3} remain true, if
in its formulation instead of the specified conditions on function
$Q$ to require that, for any $p_0\in \partial D^{\,\prime}$ there is
$\delta(p_0)>0$ such that
\begin{equation}\label{eq5**}
\int\limits_{\varepsilon}^{\delta(p_0)}
\frac{dt}{tq_{p_0}(t)}<\infty, \qquad \int\limits_{0}^{\delta(p_0)}
\frac{dt}{tq_{p_0}(t)}=\infty
\end{equation}
for sufficiently small $\varepsilon>0.$ Indeed, under the proof of
Theorem~\ref{th1}, we actually did not use any conditions on the
function $Q,$ except for~(\ref{eq7C}) and (\ref{eq7B}). However,
both of these conditions will obviously be met as soon as
conditions~(\ref{eq5**}) hold.
\end{remark}

\medskip
\begin{remark}\label{rem3}
The statements of Theorem~\ref{th1} and Theorem~\ref{th3} remain
true, if in its formulation instead of the specified conditions on
function $Q$ to require that, for any $p_0\in \partial D^{\,\prime}$
there is $\varepsilon_0=\varepsilon_0(p_0)>0$ and a Lebesgue
measurable function $\psi:(0, \varepsilon_0)\rightarrow [0, \infty]$
such that
\begin{equation}\label{eq7BA} I(\varepsilon,
\varepsilon_0):=\int\limits_{\varepsilon}^{\varepsilon_0}\psi(t)\,dt
< \infty\quad \forall\,\,\varepsilon\in (0, \varepsilon_0)\,,\quad
I(\varepsilon, \varepsilon_0)>0\quad
\text{as}\quad\varepsilon\rightarrow 0\,,
\end{equation}
and, in addition,
\begin{equation} \label{eq7CC}
\int\limits_{\widetilde{A}(p_0, \varepsilon, \varepsilon_0)}
Q(p)\cdot\psi^{\,2}(\widetilde{h_*}(p,
p_0))\,d\widetilde{v_*}(p)\leqslant C_0I^2(\varepsilon,
\varepsilon_0)\,,\end{equation}
as $\varepsilon\rightarrow 0,$ where $C_0$ is some constant, and
$\widetilde{A}(p_0, \varepsilon, \varepsilon_0)$ is defined
in~(\ref{eq1**}).

\medskip
Indeed, literally repeating the proof of the statement given in
Theorem~\ref{th1} to the ratio~(\ref{eq3}) inclusive, we put
$$\eta(t)=\left\{
\begin{array}{rr}
\psi(t)/I(r_1, r_2), & t\in (r_1, r_2)\,,\\
0,  &  t\not\in (r_1, r_2)\,,
\end{array}
\right. $$
where $I(r_1, r_2)=\int\limits_{r_1}^{r_2}\,\psi (t)\, dt.$ Observe
that
$\int\limits_{r_1}^{r_2}\eta(t)\,dt=1.$ Now, by the definition of
$f$ in~(\ref{eq2*A}) and due to the relation~(\ref{eq3}) we obtain
that
\begin{equation}\label{eq14C}
M(\Gamma(|\alpha^{\,*}_k|, |\beta^{\,*}_k|, D))\leqslant
C_0<\infty\,.
\end{equation}
The relation~(\ref{eq14C}) contradicts with~(\ref{eq7}). The
resulting contradiction proves the desired statement.

Similar arguments may be given with respect to the
Theorem~\ref{th3}.~$\Box$
\end{remark}

\section{On the discreteness of mappings at the boundary}

In~\cite{Vu}, some issues related to the discreteness of a closed
quasiregular map $f:{\Bbb B}^n\rightarrow {\Bbb R}^n$ in
$\overline{{\Bbb B}^n}$ are considered. In this section we talk
about the discreteness of mappings that satisfy the
condition~(\ref{eq2*A}) on Riemannian manifolds.

\medskip
Following~\cite[Section~2.4]{NP}, we say that a domain $D\subset
{\Bbb S}$ is {\it uniform}, if for any $r>0$ there is $\delta>0$
such that the inequality
\begin{equation}\label{eq17***}
M(\Gamma(F^{\,*},F, D))\geqslant \delta
\end{equation}
holds for any continua $F, F^*\subset D$ with
$\widetilde{h}(F)\geqslant r$ and $\widetilde{h}(F^{\,*})\geqslant
r.$ Note that this is the definition slightly different from the
''classical'' given in \cite[Chapter~2.4]{NP}, where the sets $F$
and $F^*\subset D $ are assumed to be arbitrary connected. We prove
the following statement (see its analogue for quasiregular mappings
of the unit ball in~\cite[Lemma~4.4]{Vu}).

\medskip
\begin{lemma}\label{lem1A}
{\,\sl Let $D$ and $D_*$ be domains on Riemannian surfaces ${\Bbb
S}$ and ${\Bbb S}_{\,*},$ respectively, $D$ is a uniform domain,
$\overline{D_*}$ is a compactum in ${\Bbb S}_*,$ in addition,
$\partial D_*$ has a finite number of components. Let
$f:D\rightarrow D_*$ be a mapping satisfying the
relations~(\ref{eq2*A})--(\ref{eqA2}) hold for any $p_0\in
\partial D_*.$ Assume that, for any $p_0\in
\partial D_*$ there is $\varepsilon_0=\varepsilon_0(p_0)>0$ and a Lebesgue measurable
function $\psi:(0, \varepsilon_0)\rightarrow [0,\infty]$ such that
\begin{equation}\label{eq7***} I(\varepsilon,
\varepsilon_0):=\int\limits_{\varepsilon}^{\varepsilon_0}\psi(t)\,dt
< \infty\quad \forall\,\,\varepsilon\in (0, \varepsilon_0)\,,\quad
I(\varepsilon, \varepsilon_0)\rightarrow
\infty\quad\text{as}\quad\varepsilon\rightarrow 0\,,
\end{equation}
and, in addition,
\begin{equation} \label{eq3.7.2}
\int\limits_{\widetilde{A}(p_0, \varepsilon, \varepsilon_0)}
Q(y)\cdot\psi^{\,p}(\widetilde{h_*}(p, p_0))\,d\widetilde{v_*}(p) =
o(I^p(\varepsilon, \varepsilon_0))\,,\end{equation}
as $\varepsilon\rightarrow 0,$ where $\widetilde{A}(p_0,
\varepsilon, \varepsilon_0)$ is defined in~(\ref{eq1**}). Let $C_j,$
$j=1,2,\ldots ,$ be a sequence of continua such that
$\widetilde{h}(C_j)\geqslant \delta>0$ for some $\delta>0$ and any
$j\in {\Bbb N}$ and, in addition, $\rho(f(C_j))\rightarrow 0$ as
$j\rightarrow\infty,$ where $\rho$ is some possible metric in
$\overline{D_*}_P.$ Then there is $\delta_1>0$ such that
$$\rho(f(C_j), P_0)\geqslant \delta_1>0$$
for any $j\in {\Bbb N}$ and for any $P_0\in E_{D_*},$ where the
metrics $\rho$ is defined in item~2) of Proposition~\ref{pr1}.

Here, as usually, $$\rho(A)=\sup\limits_{p, p_*\in A}\rho(p,
p_*)\,,$$
$$\rho(A, B)=\inf\limits_{p\in A, p_*\in B}\rho(p, p_*)\,.$$
}
\end{lemma}

\begin{proof}
Suppose the opposite, namely, let $\rho(f(C_{j_k}), P_0)\rightarrow
0$ as $k\rightarrow\infty$ for some $P_0\in E_{D_*}$ and for some
increasing sequence of numbers $j_k,$ $k=1,2,\ldots .$ Let $F\subset
D$ be any continuum in $D,$ and let $\Gamma_k:=\Gamma(F, C_{j_k},
D).$ Due to the definition of the uniformity of the domain, we
obtain that
\begin{equation}\label{eq3A}
M(\Gamma_k)\geqslant \delta_2>0
\end{equation}
for any $k\in {\Bbb N}$ and some $\delta_2>0.$ On the other hand,
let us to consider the family of paths~$f(\Gamma_k).$ Let $d_l,$
$l=1,2,\ldots ,$ be a sequence of domains which corresponds to the
prime end $P_0,$ and let $\sigma_l$ be a cut corresponding to $d_l.$
Due to Proposition~\ref{pr1}, we may assume that $\sigma_l,$
$l=1,2,\ldots, $ belong to circles $\widetilde{S}(p_0, r_l)$
centered at some point $p_0\in
\partial D_*,$ where $r_l\rightarrow 0$ as $l\rightarrow\infty.$

\medskip
Let us to prove that, for any $l\in {\Bbb N}$ there is a number
$k=k_l$ such that
\begin{equation}\label{eq3M}
f(C_{j_k})\subset d_l\,,\qquad k\geqslant k_l\,.
\end{equation}
Suppose the opposite. Then there is $l_0\in {\Bbb N}$ such that
\begin{equation}\label{eq3F}
f(C_{j_{m_l}})\cap ({\Bbb R}^n\setminus d_{l_0})\ne\varnothing
\end{equation}
for some increasing sequence of numbers  $m_l,$ $l=1,2,\ldots .$  In
this case, there is a sequence $x_{m_l}\in f(C_{j_{m_l}})\cap ({\Bbb
R}^n\setminus d_{l_0}),$ $l\in {\Bbb N}.$ Since by the assumption
$\rho(f(C_{j_k}), P_0)\rightarrow 0$ for some sequence of numbers
$j_k,$ $k=1,2,\ldots ,$ we obtain that
\begin{equation}\label{eq3E}
\rho(f(C_{j_{m_l}}), P_0)\rightarrow 0\qquad {\text as}\qquad
l\rightarrow\infty\,.
\end{equation}
Since $\rho(f(C_{j_{m_l}}), P_0)=\inf\limits_{y\in
f(C_{j_{m_l}})}\rho(y, P_0)$ and $f(C_{j_{m_l}})$ is a compact set
in $\overline{D_*}_P$ as a continuous image of the compactum
$C_{j_{m_l}}$ under the mapping $f,$ it follows that
$\rho(f(C_{j_{m_l}}), P_0)=\rho(y_l, P_0),$ where $y_l\in
f(C_{j_{m_l}}).$ Due to the relation~(\ref{eq3E}) we obtain that
$y_l\rightarrow p_0$ as $l\rightarrow\infty$ in the metric $\rho.$
Since by the assumption $\rho(f(C_j))=\sup\limits_{y,z\in
f(C_j)}\rho(y,z)\rightarrow 0$ as $j\rightarrow\infty,$ we have that
$\rho(y_l, x_{m_l})\leqslant \rho(f(C_{j_{m_l}}))\rightarrow 0$ as
$l\rightarrow\infty.$  Now, by the triangle inequality, we obtain
that
$$\rho(x_{m_l}, P_0)\leqslant \rho(x_{m_l}, y_l)+\rho(y_l, P_0)
\rightarrow 0\qquad {\text as}\quad l\rightarrow\infty\,.$$
The latter contradicts with~(\ref{eq3F}). The contradiction obtained
above proves~(\ref{eq3M}).

\medskip
Without loss of generality we may consider that the number $l_0\in
{\Bbb N}$ is such that $r_l<\varepsilon_0$ for any $l\geqslant l_0,$
and
\begin{equation}\label{eq3I}
f(F)\subset {\Bbb R}^n\setminus d_1\,.
\end{equation}
In this case, we observe that, for $l\geqslant 2$
\begin{equation}\label{eq3G}
f(\Gamma_{k_l})>\Gamma(\widetilde{S}(p_0, r_l), \widetilde{S}(p_0,
r_1), \widetilde{A}(p_0, r_l, r_1))\,.
\end{equation}
Indeed, let $\widetilde{\gamma}\in f(\Gamma_{k_l}).$ Then
$\widetilde{\gamma}(t)=f(\gamma(t)),$ where $\gamma\in
\Gamma_{k_l},$ $\gamma:[0, 1]\rightarrow D,$ $\gamma(0)\in F,$
$\gamma(1)\in C_{j_{k_l}}.$ Due to the relation~(\ref{eq3I}), we
obtain that $f(\gamma(0))\in f(F)\subset {\Bbb S}\setminus
\widetilde{B}(p_0, \varepsilon_0).$ On the other hand,
by~(\ref{eq3M}), $\gamma(1)\in C_{j_{k_l}}\subset d_l\subset d_1.$
Thus, $|f(\gamma(t))|\cap d_1\ne\varnothing \ne |f(\gamma(t))|\cap
({\Bbb S}\setminus d_1).$ Now, by~\cite[Theorem~1.I.5.46]{Ku} we
obtain that, there is $0<t_1<1$ such that $f(\gamma(t_1))\in
\partial d_1\cap D\subset \widetilde{S}(p_0, r_1).$ Set
$\gamma_1:=\gamma|_{[t_1, 1]}.$ We may consider that
$f(\gamma(t))\in d_1$ for any $t\geqslant t_1.$ Arguing similarly,
we obtain $t_2\in [t_1, 1]$ such that $f(\gamma(t_2))\in
\widetilde{S}(p_0, r_l).$ Put $\gamma_2:=\gamma|_{[t_1, t_2]}.$ We
may consider that $f(\gamma(t))\in d_l$ for any $t\in [t_1, t_2].$
Now, a path $f(\gamma_2)$ is a subpath of
$f(\gamma)=\widetilde{\gamma},$ which belongs to
$\Gamma(\widetilde{S}(p_0, r_l), \widetilde{S}(p_0, r_1),
\widetilde{A}(p_0,r_l, r_1)).$ The relation~(\ref{eq3G}) is
established.

\medskip
It follows from~(\ref{eq3G}) that
\begin{equation}\label{eq3H}
\Gamma_{k_l}>\Gamma_{f}(\widetilde{S}(p_0, r_l), \widetilde{S}(p_0,
r_1), \widetilde{A}(p_0, r_l, r_1))\,.
\end{equation}
Set
$$\eta_{l}(t)=\left\{
\begin{array}{rr}
\psi(t)/I(r_l, r_1), & t\in (r_l, r_1)\,,\\
0,  &  t\not\in (r_l, r_1)\,,
\end{array}
\right. $$
where $I(r_l, r_1)=\int\limits_{r_l}^{r_1}\,\psi (t)\, dt.$ Observe
that
$\int\limits_{r_l}^{r_1}\eta_{l}(t)\,dt=1.$ Now, by the
relations~(\ref{eq3.7.2}) and~(\ref{eq3H}), and due to the
definition of $f$ in~(\ref{eq2*A}), we obtain that
$$M_p(\Gamma_{k_l})\leqslant M_p(\Gamma_{f}(\widetilde{S}(p_0, r_l), \widetilde{S}(p_0,
r_1), \widetilde{A}(p_0, r_l, r_1)))\leqslant$$
\begin{equation}\label{eq3J}
\leqslant \frac{1}{I^p(r_l, r_1)}\int\limits_{\widetilde{A}(p_0,
r_l, r_1)} Q(y)\cdot\psi^{\,p}(\widetilde{h_*}(p,
p_0))\,d\widetilde{v_*}(p)\rightarrow 0\quad \text{as}\quad
l\rightarrow\infty\,.
\end{equation}
The relation~(\ref{eq3J}) contradicts with~(\ref{eq3A}). The
contradiction obtained above proves the lemma.~$\Box$
\end{proof}

The analog of the following lemma was proved
in~\cite[Corollary~4.5]{Vu}.

\medskip
\begin{lemma}\label{lem4B}
{\,\sl Let $D$ and $D_*$ be domains on Riemannian surfaces ${\Bbb
S}$ and ${\Bbb S}_{\,*},$ respectively, having compact closures
$\overline{D}\subset {\Bbb S}$ and $\overline{D_*}\subset {\Bbb
S}_*,$ in addition, $\partial D$ and $\partial D_*$ have a finite
number of components. Let $f$ be an open discrete and closed mapping
of $D$ onto $D_*$ satisfying the
relations~(\ref{eq2*A})--(\ref{eqA2}) for any $p_0\in
\partial D_*.$ Assume that, for any $p_0\in
\partial D_*$ there is $\varepsilon_0=\varepsilon_0(p_0)>0$ and a Lebesgue measurable
function $\psi:(0, \varepsilon_0)\rightarrow [0,\infty]$ such
that~(\ref{eq7***})--(\ref{eq3.7.2}) as $\varepsilon\rightarrow 0,$
where $\widetilde{A}(p_0, \varepsilon, \varepsilon_0)$ is defined
in~(\ref{eq1**}).

Then the mapping $f$ has a continuous extension
$\overline{f}:\overline{D}_P\rightarrow \overline{D_*}_P$ which is
light, i.e., any component of $f^{\,-1}(\widetilde{P_0})$ of
$\widetilde{P_0}\in E_{D_*}$ degenerate into a point.}
\end{lemma}

\begin{proof}
The possibility of continuous extension of $f$ to a mapping
$\overline{f}:\overline{D}_P\rightarrow \overline{D_*}_P$ follows by
Remark~\ref{rem3}. Assume the contrary, i.e., there is a continuum
$K\subset E_D$ and a point $y_0\in E_{D_*}$ such that
$\overline{f}(K)=y_0.$ By Remark~2 in \cite{RV$_2$}, any two prime
ends $P_1$ and $P_2$ which correspond to different components
$\Gamma_1$ and $\Gamma_2$ of the boundary $\partial D$ are contained
in disjoint neighborhoods $U_1, U_2\subset \overline{D}_P.$ Thus,
the continuum $K$ corresponds to one and only one component
$\Gamma_0\subset \partial D.$

\medskip
By Proposition~\ref{pr1}, there is a neighborhood $U_0$ of
$\Gamma_0$ such that, there is a conformal mapping $H_0$ of $U_0$
onto some ring $A(0, r_0, 1),$ $0<r_0<1,$ such that
$\gamma_0:=\partial U_0\cap D$ is a closed Jordan path and
$$C(H_0, \gamma_0) = \{z\in {\Bbb C}: |z| = 1\};\quad  C(H_0, \Gamma_0) =
\{z\in {\Bbb C}: |z| = r_0\}\,.$$
Set $$F(x):=(\overline{f}\circ H^{\,-1}_0)(x)\,, x\in A(0, r_0,
1)\,.$$
Arguing similarly to the proof of Theorem~\ref{th1}, we may show
that, $F$ satisfies~(\ref{eq2*A}) in $A(0, r_0, 1),$ as well.

\medskip
Observe that, $A(0, r_0, 1)$ is a uniform domain as a domain as a
plane domain with a finite components of the boundary which is
finitely connected on the boundary (see
\cite[Corollary~6.8]{Na$_1$}). Taking into account the simple
geometry of the ring $A(0, r_0, 1)$ and, in addition, that $F$ is
continuous in $\overline{A(0, r_0, 1)},$ we may construct the
sequence of continua $C_m\subset A(0, r_0, 1),$ $m=1,2,\ldots,$
converging to $H^{\,-1}_0(K_1)$ in the sense of Hausdorff distance
such that $d(C_m)\geqslant d(K)/2$ and $f(C_m)\rightarrow y_0$ as
$m\rightarrow\infty.$ (Here $d(A)$ denotes the Euclidean diameter of
the set $A$). The latter contradicts with Lemma~\ref{lem1A}. The
contradiction obtained above proves the Lemma.~$\Box$
\end{proof}

Let $p_0\in {\Bbb S}$ and let $\varphi:{\Bbb S}\rightarrow {\Bbb R}$
be a function integrable in some neighborhood $U$ of the point $p_0$
with respect to $\widetilde{v}.$ Following~\cite[Section~6.1,
Ch.~6]{MRSY}, we say that a function $\varphi:{\Bbb S}\rightarrow
{\Bbb R}$ has a {\it finite mean oscillation} at the point $p_0\in
D$, we write $\varphi\in FMO (p_0),$ if
%
%
%
%\begin{equation}\label{eq29*!}
%
$$\limsup\limits_{\varepsilon\rightarrow
0}\frac{1}{\widetilde{v}(\widetilde{B}(p_0,
\varepsilon))}\int\limits_{\widetilde{B}(p_0,\,\varepsilon)}
|{\varphi}(p)-\overline{\varphi}_{\varepsilon}|\
d\widetilde{v}(p)<\infty\,,$$
%
%\end{equation}
%
where
$\overline{{\varphi}}_{\varepsilon}=\frac{1}
{\widetilde{v}(\widetilde{B}(p_0,
\varepsilon))}\int\limits_{\widetilde{B}(p_0, \varepsilon)}
{\varphi}(p) \,d\widetilde{v}(p).$ The following statement holds.

\medskip
\begin{theorem}\label{th2}
{\,\sl Let $D$ and $D_*$ be domains on Riemannian surfaces ${\Bbb
S}$ and ${\Bbb S}_{\,*},$ respectively, having compact closures
$\overline{D}\subset {\Bbb S}$ and $\overline{D_*}\subset {\Bbb
S}_*,$ in addition, $\partial D$ and $\partial D_*$ have a finite
number of components. Let $f$ be an open discrete and closed mapping
of $D$ onto $D_*$ satisfying the
relations~(\ref{eq2*A})--(\ref{eqA2}) for any $p_0\in
\partial D_*.$ Assume that, either of the following conditions hold:

\medskip
1) $Q\in FMO(\partial D_*),$

\medskip
2) the relations
\begin{equation}\label{eq45}
\int\limits_{\varepsilon}^{\varepsilon_0}
\frac{dt}{tq_{p_0}(t)}<\infty\,,\qquad \int\limits_0^{\varepsilon_0}
\frac{dt}{q_{p_0}(r)}=\infty\,,
\end{equation}
hold for any $p_0\in \partial  D_*,$ some $\varepsilon_0>0$ and any
$\varepsilon\in (0, \varepsilon_0,$ where $q_{p_0}(t)$ is defined
in~(\ref{eq26}).

Then the mapping $f$ has a continuous extension
$\overline{f}:\overline{D}_P\rightarrow \overline{D_*}_P$ which is
light, i.e., any component of $f^{\,-1}(\widetilde{P_0})$ of
$\widetilde{P_0}\in E_{D_*}$ degenerate into a point.}
\end{theorem}

\begin{proof}
Proof immediately follows from Lemma~\ref{lem4B}, taking into
account that the relations $Q\in FMO(\partial D_*)$ or~(\ref{eq45})
are particular cases of~(\ref{eq7***})--(\ref{eq3.7.2}) (see
Proposition~4.7 and details of proof of Theorem~4.8
in~\cite{Sev$_6$}).~$\Box$
\end{proof}

Given a mapping $f:D\,\rightarrow\,{\Bbb S},$ a set $E\subset D$ and
$y\,\in\,{\Bbb S},$ we define the {\it multiplicity function
$N(y,f,E)$} as a number of preimages of the point $y$ in a set $E,$
i.e.
$$
N(y,f,E)\,=\,{\rm card}\,\left\{x\in E: f(x)=y\right\}\,,
$$
\begin{equation}\label{eq1G}
N(f,E)\,=\,\sup\limits_{y\in{\Bbb S}}\,N(y,f,E)\,.
\end{equation}
Note that, the concept of a multiplicity function may also be
extended to sets belonging to the closure of a given domain.
Finally, we formulate and prove a key statement about the
discreteness of mapping (see~\cite[Theorem~4.7]{Vu}).

\medskip
\begin{lemma}\label{lem3A}
{\,\sl Let $D$ and $D_*$ be domains on Riemannian surfaces ${\Bbb
S}$ and ${\Bbb S}_{\,*},$ respectively, having compact closures
$\overline{D}\subset {\Bbb S}$ and $\overline{D_*}\subset {\Bbb
S}_*,$ in addition, $\partial D_*$ has a finite number of
components. Assume that, $D$ has a weakly flat boundary. Let $f$ be
an open discrete and closed mapping of $D$ onto $D_*$ satisfying the
relations~(\ref{eq2*A})--(\ref{eqA2}) for any $p_0\in
\partial D_*.$ Assume that, for any $p_0\in
\partial D_*$ there is $\varepsilon_0=\varepsilon_0(p_0)>0$ and a Lebesgue measurable
function $\psi:(0, \varepsilon_0)\rightarrow [0,\infty]$ such
that~(\ref{eq7***})--(\ref{eq3.7.2}) as $\varepsilon\rightarrow 0,$
where $\widetilde{A}(p_0, \varepsilon, \varepsilon_0)$ is defined
in~(\ref{eq1**}).

Then the mapping $f$ has a continuous extension
$\overline{f}:\overline{D}\rightarrow \overline{D_*}_P$ such that
$N(f, D)=N(f, \overline {D})<\infty.$ In particular, $\overline{f}$
is discrete in $\overline{D},$ that is, $\overline{f}^{\,-1}(P_0)$
consists only from isolated points for any $P_0\subset E_{D_*}.$ }
\end{lemma}

\begin{proof}
First of all, the possibility of continuous extension of $f$ to a
mapping $\overline{f}:\overline{D}\rightarrow \overline{D_*}_P$
follows by Theorem~\ref{th3} together with Remark~\ref{rem3}. Note
also that, $N(f, D) <\infty$ (see~\cite[Theorem~2.8]{MS}). Let us to
prove that $N(f, D)=N(f, \overline{D}).$ Next we will reason using
the scheme proof of Theorem~4.7 in \cite{Vu}. Let $k:=N(f, D).$

\medskip
Assume the contrary, i.e., $N(f, D)<N(f, \overline {D}).$ Then there
are $\widetilde{P_0}\in E_{D_*}$ and points $x_1,x_2,\ldots, x_k,
x_{k+1}\in \overline{D}$ such that $f(x_m)=\widetilde{P_0},$
$m=1,2,\ldots .$ Since $f$ is closed, $x_m\in \partial{D}$ (see
\cite[Proposition~4.1]{Sev$_6$}).

\medskip
Set
$$\widetilde{B}_{\rho}(\widetilde{P_0}, 1/p):=\{y\in \overline{D_*}_P:
\rho(y, \widetilde{P_0})<1/p\}\,,$$
where $\rho$ is a metric in $\overline{D_*}_P.$ Let $d_0^k,$
$k=1,2,\ldots,$ be a sequence of domains corresponding to
$\widetilde{P_0}.$ Now, by the definition of the sequence $d_0^k,$
there is a number $k_{0,p}\in {\Bbb N}$ such that $d_0^{k_{0,
p}}\subset \widetilde{B}_{\rho}(P_0, 1/p).$ Set
$$U^{\,\prime}_p:=d_0^{k_{0,
p}}.$$

\medskip
Let us to prove that, for any $i=1,2,\ldots, k+1$ there is a
component $V_p^i$ of the set $f^{\,-1}(U^{\,\prime}_p)$ such that
$x_i\in\overline{V_p^i}.$ Fix $i=1,2,\ldots, k+1.$ By the continuity
of $f$ in $D,$ there is $r_i=r_i(x_i)>0$ such that
$f(\widetilde{B}(x_i, r_i)\cap D)\subset U^{\,\prime}_p,$ where
$$\widetilde{B}(x_i, r_i):=\{x\in {\Bbb S}: \widetilde{h}(x, x_i)<r_i\}\,.$$
Since $W_i:=\widetilde{B}(x_i, r_i)\cap D$ is connected, $W_i$
belongs to one and only one component $V^p_i$ of the set
$f^{\,-1}(U^{\,\prime}_p),$ while $x_i\in\overline{W_i}\subset
\overline{V_p^i},$ as required.

\medskip
Next we show that the sets $\overline{V_p^i}$ are disjoint for any
$i=1,2,\ldots, k+1$ and large enough $p\in {\Bbb N}.$ In turn, we
prove for this that $\widetilde{h}(\overline{V_p^i})\rightarrow 0$
as $p\rightarrow\infty$ for each fixed $i=1,2,\ldots, k+1,$ where
$\widetilde{h}(A)$ denotes the diameter of the set $A$ in ${\Bbb S}$
with respect to the metric $\widetilde{h}.$ Let us prove the
opposite. Then there is $1\leqslant i_0\leqslant k+1,$ a number
$r_0>0,$ $r_0<\frac{1}{2}\min\limits_{1\leqslant i, j\leqslant k+1,
i\ne j}\widetilde{h}(x_i, x_j)$ and an increasing sequence of
numbers $p_m,$ $m=1,2,\ldots,$ such that $\widetilde{S}(x_{i_0},
r_0)\cap \overline{V_{p_m}^{i_0}}\ne\varnothing,$ where
$\widetilde{S}(x_0, r)=\{x\in {\Bbb S}: \widetilde{h}(x, x_0)=r\}.$
In this case, there are $a_m, b_m\in V_{p_m}^{i_0}$ such that
$a_m\rightarrow x_{i_0}$ as $m\rightarrow\infty$ and
$\widetilde{h}(a_m, b_m)\geqslant r_0/2.$ Join the points $a_m$ and
$b_m$ by a path $C_m,$ which entirely belongs to $V_{p_m}^{i_0}.$
Then $\widetilde{h}(|C_m|)\geqslant r_0/2$ for $m=1,2,\ldots .$ On
the other hand, since $|C_m|\subset f(V_{p_m}^{i_0})\subset
\widetilde{B}_{\rho}(\widetilde{P_0}, 1/p_m),$ then simultaneously
$\rho(f(|C_m|))\rightarrow 0$ as $m\rightarrow\infty$ and
$\rho(f(|C_m|), p_0)\rightarrow 0$ as $m\rightarrow\infty.$ Since
$D$ has a weakly flat boundary (see e.g.~\cite[Theorems~17.10,
17.12]{Va}), the latter contradicts with Lemma~\ref{lem1A}. The
resulting contradiction indicates the incorrectness of the above
assumption.

By~\cite[Lemma~2.2]{MS}, $f$ is a mapping of $\overline{V_p^i}$ onto
$U^{\,\prime}_p$ for any $i=1,2,\ldots, k, k+1.$ Thus, $N(f, D)=N(f,
D)\geqslant k+1,$ which contradicts the definition of the number
$k.$ The obtained contradiction refutes the assumption that $N(f,
\overline{D})>N(f, D).$ The lemma is proved.~$\Box$
\end{proof}

\medskip
\begin{theorem}\label{th4}
{\,\sl Let $D$ and $D_*$ be domains on Riemannian surfaces ${\Bbb
S}$ and ${\Bbb S}_{\,*},$ respectively, having compact closures
$\overline{D}\subset {\Bbb S}$ and $\overline{D_*}\subset {\Bbb
S}_*,$ in addition, $\partial D_*$ has a finite number of
components. Assume that, $D$ has a weakly flat boundary. Let $f$ be
an open discrete and closed mapping of $D$ onto $D_*$ satisfying the
relations~(\ref{eq2*A})--(\ref{eqA2}) for any $p_0\in
\partial D_*.$  Assume that, either of the following conditions hold:

\medskip
1) $Q\in FMO(\partial D_*),$

\medskip
2) the relations
\begin{equation}\label{eq45A}
\int\limits_{\varepsilon}^{\varepsilon_0}
\frac{dt}{tq_{p_0}(t)}<\infty\,,\qquad \int\limits_0^{\varepsilon_0}
\frac{dt}{q_{p_0}(r)}=\infty\,,
\end{equation}
hold for any $p_0\in \partial  D_*,$ some $\varepsilon_0>0$ and any
$\varepsilon\in (0, \varepsilon_0,$ where $q_{p_0}(t)$ is defined
in~(\ref{eq26}).

Then the mapping $f$ has a continuous extension
$\overline{f}:\overline{D}\rightarrow \overline{D_*}_P$ such that
$N(f, D)=N(f, \overline {D})<\infty.$ In particular, $\overline{f}$
is discrete in $\overline{D},$ that is, $\overline{f}^{\,-1}(P_0)$
consists only from isolated points for any $P_0\subset E_{D_*}.$ }
\end{theorem}

\begin{proof}
Proof immediately follows from Lemma~\ref{lem3A}, taking into
account that the relations $Q\in FMO(\partial D_*)$ or~(\ref{eq45A})
are particular cases of~(\ref{eq7***})--(\ref{eq3.7.2}) (see
Proposition~4.7 and details of proof of Theorem~4.8
in~\cite{Sev$_6$}).~$\Box$
\end{proof}

\medskip
{\bf \noindent Evgeny Sevost'yanov} \\
{\bf 1.} Zhytomyr Ivan Franko State University,  \\
40 Bol'shaya Berdichevskaya Str., 10 008  Zhytomyr, UKRAINE \\
{\bf 2.} Institute of Applied Mathematics and Mechanics\\
of NAS of Ukraine, \\
1 Dobrovol'skogo Str., 84 100 Slavyansk,  UKRAINE\\
esevostyanov2009@gmail.com

\end{document}